\documentclass[12pt]{article}
\usepackage{amsmath}
\usepackage{amssymb}
\usepackage{amsthm}
\usepackage[mathscr]{eucal}
\usepackage{latexsym}
\usepackage{amscd}

\begin{document}
\title{\textbf{Stability of Kronecker coefficients via discrete tomography}}
\author{{\sf Ernesto Vallejo}}
\date{\today}

\voffset -1.5truecm
\oddsidemargin0.5truecm
\evensidemargin0.5truecm

\theoremstyle{plain}
\swapnumbers\newtheorem{lema}{Lemma}[section]
\newtheorem{prop}[lema]{Proposition}
\newtheorem{coro}[lema]{Corollary}
\newtheorem{teor}[lema]{Theorem}
\newtheorem{demo}[lema]{Proof of Theorem 1}
\newtheorem{mteor}[lema]{Theorem (Additive stability)}

\theoremstyle{definition}
\newtheorem{defi}[lema]{Definition}
\newtheorem{defis}[lema]{Definitions}
\newtheorem{ejem}[lema]{Example}
\newtheorem{ejems}[lema]{Examples}
\newtheorem{obse}[lema]{Remark}
\newtheorem{obses}[lema]{Remarks}
\newtheorem{nota}[lema]{Notation}
\newtheorem{intmat}[lema]{Integral matrices}
\newtheorem{binmat}[lema]{Binary matrices}
\newtheorem{morintmat}[lema]{More on integral matrices}
\newtheorem{simetria}[lema]{Symmetries of plane partitions}
\newtheorem{obstru}[lema]{Obstructions to additivity}
\newtheorem{complejidad}[lema]{Complexity}
\newtheorem{getepat}[lema]{Gelfand-Tsetlin patterns}
\newtheorem{bounds}[lema]{Bounds for stability}

\renewcommand{\refname}{\large\bf References}

\def\natural{{\mathbb N}}
\def\noneg{\natural_0}
\def\entero{{\mathbb Z}}
\def\racional{{\mathbb Q}}
\def\real{{\mathbb R}}
\def\complejo{{\mathbb C}}
\def\planod{\natural \times \natural}
\def\vacio{\varnothing}

\def\flecha{\longrightarrow}
\def\asocia{\longmapsto}
\def\sii{\Longleftrightarrow}
\def\vector#1#2{({#1}_1,\dots,{#1}_{#2})}
\def\conjunto#1{\boldsymbol{[}\,{#1}\,\boldsymbol{]}}
\def\particion{\vdash}
\def\implica{\Longrightarrow}
\def\rimplica{\Longleftarrow}
\def\ol#1{\overline{#1}}
\def\fun#1#2#3{{#1}\,\colon {#2} \flecha {#3}}
\def\nume#1{\boldsymbol{[}\,{#1}\,\boldsymbol{]}}
\def\negra#1{\boldsymbol{#1}}
\def\bili#1#2{\langle {#1}, {#2} \rangle}
\def\bilipunto{\bili{\,\cdot\,}{\cdot\,}}
\def\wt#1{\widetilde{#1}}
\def\matricita#1#2#3#4{ \left[\begin{smallmatrix} #1 & #2 \\ #3 & #4
                           \end{smallmatrix}\right] }

\def\sime#1{{\sf S}_{#1}}

\def\cara#1{\chi^{#1}}
\def\permu#1{\phi^{\,#1}}
\def\coefi{{\sf g}(\la,\mu,\nu)}
\def\coefili#1#2#3{{\sf g}({#1},{#2},{#3})}
\def\kron{\chi^\lambda\otimes\chi^\mu}
\def\kronli#1#2{\chi^{#1}\otimes\chi^{#2}}

\def\la{\lambda}
\def\longi#1{\ell({#1})}
\def\prof#1{{\sf d}({#1})}

\def\domina{\trianglerighteq}
\def\dominada{\trianglelefteq}
\def\nodominada{\ntrianglelefteq}

\def\menore{\prec}
\def\menori{\preccurlyeq}
\def\mayore{\succ}
\def\mayori{\succcurlyeq}

\def\abc#1#2#3{{#1}_{{#2}{#3}}}

\def\contenido#1{{\sf cont}({#1})}
\def\forma#1{{\sf sh}({#1})}
\def\cano#1{{\sf C}({#1})}
\def\insertar#1#2#3{{#1}_{#2} \rightarrow ( \cdots  \rightarrow
({#1}_1 \rightarrow {#3}) \cdots )}
\def\kos#1#2{K_{{#1}{#2}}}
\def\kostkasf#1#2{{\sf K}_{{#1}{#2}}}
\def\gtpat#1#2{{\sf ST}({#1},{#2})}
\def\gtele#1#2#3{({#1}_{{#2},{#3}})}
\def\suce#1{({#1}_{c(1)}, \dots, {#1}_{c(pq)})}

\def\lr{{\sf LR}(\alpha,\beta;\nu)}
\def\lrest{{\sf LR}^*(\alpha,\beta;\nu)}
\def\lrd{{\sf LR}^*(\alpha,\beta;\nu)}

\def\dmat{{\sf M}(\lambda,\mu)}
\def\dmatnum{{\sf m}(\lambda,\mu)}
\def\dmata{{\sf M}(\alpha,\beta)}
\def\dmatanum{{\sf m}(\alpha,\beta)}
\def\dmatapi{{\sf M}(\alpha,\beta)_{\gamma}}
\def\dmatapinum{{\sf m}(\alpha,\beta)_{\gamma}}
\def\dmatpili#1#2#3{{\sf M}({#1},{#2})_{#3}}
\def\dmatli#1#2{{\sf M}({#1},{#2})}
\def\tmats{{\sf M}^*(\lambda,\mu,\nu)}
\def\tmatsnum{{\sf m}^*(\lambda,\mu,\nu)}
\def\tmata{{\sf M}^*(\alpha,\beta,\gamma)}
\def\tmatanum{{\sf m}^*(\alpha,\beta,\gamma)}
\def\tmatsc{{\sf M}^*(\lambda,\mu,\nu^{\,\prime})}
\def\tmatscnum{{\sf m}^*(\lambda,\mu,\nu^{\,\prime})}
\def\tmatac{{\sf M}^*(\alpha,\beta,\gamma^\prime)}
\def\tmatacnum{{\sf m}^*(\alpha,\beta,\gamma^\prime)}
\def\tmatali#1#2#3{{\sf M}^*({#1},{#2},{#3})}
\def\tmatalinum#1#2#3{{\sf m}^*({#1},{#2},{#3})}
\def\renglon#1{{\sf row}(#1)}
\def\columna#1{{\sf col}({#1})}
\def\parti#1{\pi(#1)}
\def\grafi#1{{\sf G}({#1})}

\def\flat#1#2{{\sf F}({#1}, {#2})}
\def\matriz{{\sf M}(\negra u,\negra v)}
\def\permuta#1{{\sf P}({#1})}
\def\trapo{{\sf T}(\alpha, \beta)}
\def\transport#1#2{{\sf T}({#1},{#2})}
\def\politopo{\permuta A \cap \fimap \flat}
\def\fimap#1{\Phi({#1})}
\def\fimapinv#1{\Phi^{-1}({#1})}
\def\allmatrices{{\sf M}_{p,\, q}}

\def\partipla{{\sf P}_\nu(\la,\mu)}
\def\matrizd{{\sf M}^*(\lambda,\mu)}
\def\numatrizd{{\sf m}^*(\lambda,\mu)}
\def\matriztres{{\sf M}(\lambda,\mu,\nu)}
\def\numatres{{\sf m}(\lambda,\mu,\nu)}

\def\liri#1#2#3{{\sf lr}({#1},{#2};{#3})}
\def\lirid#1#2#3{{\sf lr}^*({#1},{#2};{#3})}

\def\palabra#1#2#3{{#1}_{#2}\cdots{#1}_{#3}}

\def\sst{semistandard\ tableau\ }
\def\sucesion#1#2{{#1}_1 < {#1}_2 < \cdots < {#1}_{#2}}
\def\wcol#1{w_{\rm col}({#1})}

\def\bili#1#2{\langle{#1}, {#2}\rangle}


\setlength\unitlength{0.08em}
\savebox0{\rule[-2\unitlength]{0pt}{10\unitlength}%
\begin{picture}(10,10)(0,2)
\put(0,0){\line(0,1){10}}
\put(0,10){\line(1,0){10}}
\put(10,0){\line(0,1){10}}
\put(0,0){\line(1,0){10}}
\end{picture}}

\newlength\cellsize \setlength\cellsize{18\unitlength}
\savebox2{%
\begin{picture}(18,18)
\put(0,0){\line(1,0){18}}
\put(0,0){\line(0,1){18}}
\put(18,0){\line(0,1){18}}
\put(0,18){\line(1,0){18}}
\end{picture}}
\newcommand\cellify[1]{\def\thearg{#1}\def\nothing{}%
\ifx\thearg\nothing
\vrule width0pt height\cellsize depth0pt\else
\hbox to 0pt{\usebox2\hss}\fi%
\vbox to 18\unitlength{
\vss
\hbox to 18\unitlength{\hss$#1$\hss}
\vss}}
\newcommand\tableau[1]{\vtop{\let\\=\cr
\setlength\baselineskip{-16000pt}
\setlength\lineskiplimit{16000pt}
\setlength\lineskip{0pt}
\halign{&\cellify{##}\cr#1\crcr}}}
\savebox3{%
\begin{picture}(15,15)
\put(0,0){\line(1,0){15}}
\put(0,0){\line(0,1){15}}
\put(15,0){\line(0,1){15}}
\put(0,15){\line(1,0){15}}
\end{picture}}
\newcommand\expath[1]{%
\hbox to 0pt{\usebox3\hss}%
\vbox to 15\unitlength{
\vss
\hbox to 15\unitlength{\hss$#1$\hss}
\vss}}

\maketitle

\begin{abstract}
In this paper we give a new sufficient condition for a general stability
of Kronecker coefficients, which we call it additive stability.
It was motivated by a recent talk of J.~Stembridge at the conference in honor
of Richard P. Stanley's 70th birthday, and it is based on work of the
author on discrete tomography along the years.
The main contribution of this paper is the discovery of the connection
between additivity of integer matrices and stability of Kronecker coefficients.
Additivity, in our context, is a concept from discrete tomography.
Its advantage is that it is very easy to produce lots of examples of
additive matrices and therefore of new instances of stability properties.
We also show that Stembridge's hypothesis and additivity are closely related,
and prove that all stability properties of Kronecker coefficients
discovered before fit into additive stability.

\smallskip
AMS subject classification: 05E10, 20C30, 05E05.

{\em Key Words}: Kronecker coefficient, Kostka number, Schur function, Stability,
Discrete tomography, Additivity, Transportation polytope, Plane Partition.
\end{abstract}

\section{Introduction}
\hfill

Let $\cara\la$ be the irreducible character of the symmetric group $\sime m$
associated to the partition $\la$ of $m$.
It is a major open problem in the representation theory of the symmetric group
in characteristic 0 to find a combinatorial or geometric description
(such as the existing for the Littlewood-Richardson coefficients)
of the multiplicity (usually called Kronecker coefficient)
\begin{equation} \label{ecua:kron}
\coefi = \langle \kron, \cara\nu \rangle
\end{equation}
of $\cara\nu$ in the Kronecker product $\kron$ of $\cara\la$ and $\cara\mu$.
Here, $\langle \cdot, \cdot \rangle$ denotes the
scalar product of complex characters.
Seventy six years ago Murnaghan~\cite{mur} published the first paper on the subject.
There, he stated without proof the following stability property for Kronecker
coefficients (see also~\cite{mur2}):
For each triple of partitions $\la$, $\mu$, $\nu$ of the same size,
there is a positive integer $L$, such that for all $n\ge L$,
\begin{equation} \label{ecua:mur}
\coefili {\la + (n)}{\mu + (n)}{\nu + (n)} = \coefili {\la + (L)}{\mu + (L)}{\nu + (L)}.
\end{equation}
It was first proved by Littlewood~\cite{lit2}.
Since then many proofs of this property have appeared~\cite{bri, cheb, pp3, thi, vestab1, vdiag}.
Estimations of lower bounds $L$ for stability can be found
in~\cite{bor, bri, vestab1, vdiag}.

Recently, a generalization of Murnaghan's stability was discovered~\cite[Thm.~10.2]{vdiag}.
Let $\prof\nu = |\nu| - \nu_1$ denote the \emph{depth} of $\nu$.
There it is shown that if $i \le \min\{\ell(\la), \ell(\mu)\}$,
$\la_i - \la_{i+1} \ge \prof\nu$ and $\mu_i - \mu_{i+1} \ge \prof\nu$ (if $i = \ell(\la)$,
we define $\la_{i+1} = 0$), then for all $n\in \natural$, we have
\begin{equation} \label{ecua:mur-gen}
\coefili {\la + (n^i)} {\mu + (n^i)} {\nu + (ni)} = \coefi.
\end{equation}
The case $i=1$ yields Murnaghan's stability.
Another proof of equation~\eqref{ecua:mur-gen} (with different bounds) can be
found in~\cite{pp3}.

The proof of Theorem~10.2 in~\cite{vdiag} also yields a new more general type of
stability, which includes~\eqref{ecua:mur-gen} and appears here for the first time.
We now describe it.
Let $\la$, $\mu$, $\nu$ be a triple of partitions of the same size and let
$d = \prof \nu$, then for any pair of vectors $\vector it$, $\vector nt \in \natural^t$,
such that $i_1 > i_2 > \cdots > i_t$ and $n_j \ge d$, for all $j\in \nume t$,
\begin{multline} \label{ecua:mur-mas-gen}
\textstyle
{\sf g}\left( \la + \sum_{j=1}^t ({n_j}^{i_j}),\, \mu + \sum_{j=1}^t ({n_j}^{i_j}),\,
\nu + \left( \sum_{j=1}^t n_j i_j \right) \right) \\
\textstyle
= {\sf g}\left( \la + \sum_{j=1}^t (d^{\,i_j}),\, \mu + \sum_{j=1}^t (d^{\,i_j}),\,
\nu + \left( \sum_{j=1}^t d i_j \right) \right)
\end{multline}
Let $\alpha$ be the conjugate partition of $\vector it$, then
$\sum_{j=1}^t (1^{i_j}) = \alpha$ and $\sum_{j=1}^t (d^{\,i_j}) = d \alpha$.
Thus, identity~\eqref{ecua:mur-mas-gen} can be rewritten in the following way
\begin{multline} \label{ecua:mur-mas-gen-bis}
\textstyle
{\sf g}\left( \la + \sum_{j=1}^t ({n_j}^{i_j}),\, \mu + \sum_{j=1}^t ({n_j}^{i_j}),\,
\nu + \left( \sum_{j=1}^t n_j i_j \right) \right) \\
= \coefili {\la + d\alpha}{\mu + d\alpha}{\nu + (d|\alpha|)}.
\end{multline}

Another, apparently different, kind of stability was discovered sometime ago
in~\cite[Thm.~3.1]{estab2}.
There, it is shown that, if $q$, $r$ are positive integers such that
$\longi{\la} \le qr$, $\longi{\mu} \le q$ and $\longi\nu \le r$.
Then for all $n\in \natural$,
\begin{equation} \label{ecua:estab-cubi}
\coefili {\la + (n^{qr})} {\mu + ((nr)^q)} {\nu + ((nq)^r)} = \coefi .
\end{equation}
The proof in~\cite{estab2} uses character theory of the symmetric group.
A different proof of~\eqref{ecua:estab-cubi}, based on the representation
theory of general linear group, can be found in~\cite{mani}.

Even more recently, J. Stembridge announced the following general stability
property for Kronecker coefficients.\footnote{It was presented on June 24, 2014,
in the conference in honor of Richard P. Stanley's 70th birthday.
Based on Stembridge's presentation I prepared this manuscript with the goal
to look at his generalized stability from the point of view of discrete
tomography.
After this paper was finished, I received a copy of~\cite{stem1}.
I'll add along the paper remarks in the form of footnotes explaining the
relation between some results in~\cite{stem1} and others presented here.}
Denote $\permu\la = {\rm Ind}_{\sime \la}^{\sime n}( 1_{\la})$
(see Section~\ref{sec:cara} for details).
Let $\alpha$, $\beta$, $\gamma$ be partitions of the same size, such that
$\coefili \alpha \beta \gamma >0$ and
\begin{equation} \label{ecua:stem}
\bili{\permu{n\alpha} \otimes \permu{n\beta}}{\cara{n\gamma}} = 1,
\text{ for all } n\in \natural.
\end{equation}
Then, for any triple of partitions $\la$, $\mu$, $\nu$ of the same size,
there is a positive integer $L$ such that, for all $n\ge L$,
\begin{equation} \label{ecua:gen-estab-stem}
\coefili{\la + n\alpha}{\mu + n\beta}{\nu + n\gamma} =
\coefili{\la + L\alpha}{\mu + L\beta}{\nu + L\gamma}.\footnote{This stability
property is the main result of Section~6 in~\cite{stem1}.}
\end{equation}

In this paper we will show that practically all these stability properties
fit into what we call \emph{additive stability}.
The main contribution of this paper is the discovery of the connection
between additivity of integer matrices and stability of Kronecker coefficients.
Additivity, in our context, is a concept coming from discrete tomography
(see~\cite{heka1, heka2} for an overview of the area).
In~\cite{flrs, fishe} Fishburn et al. consider the problem of studying finite
subsets in $\real^n$ that are uniquely reconstructible from its
coordinate projection counting functions.
They called them \emph{sets of uniqueness}, and observed that
there was no loss of generality in considering them as subsets of $\natural^n$.
They also introduced that notion of \emph{additive} sets and proved
that it was sufficient for uniqueness, but not necessary.
Later, the author of this paper and Torres-Ch\'azaro considered the case $n=3$,
and, by viewing those sets as 3-dimensional binary matrices, gave a characterization
of sets of uniqueness using the dominance order of partitions~\cite{tova}.
Afterward, the author of this paper introduced a notion of additivity for matrices
with nonnegative integer entries~\cite{aditivo} and used it to give a characterization
of additive sets in the sense of~\cite{flrs}.
The first proof for our characterization of uniqueness~\cite{tova} used
character theory of the symmetric group (this proof can be translated into a
purely combinatorial one~\cite{bru}).
At that time, we thought to apply the representation theory of the symmetric group
to learn about uniqueness.
It turned the other way round.
Direct applications of discrete tomography to Kronecker coefficients
were given in~\cite{avva, vapp}.
These are, to my best knowledge, the first applications of discrete tomography
to Kronecker coefficients.
Here, we provide another application: we use the notion of additivity of
integer matrices to prove a general stability for Kronecker coefficients.

We apply here ideas from~\cite{onva, sava, tova, vapir, vapp, aditivo, minidosq, elenot, libro}
to prove additive stability, give examples and get some related results.
In order to state our main theorem we need the following definition, which is fundamental
for this paper.
A matrix $A =(a_{i,j})$ of size $p\times q$ with nonnegative integer entries is called
\emph{additive} (see~\cite[Thm.~1]{aditivo}), if there exist real numbers
$x_1, \dots, x_p$, $y_1, \dots, y_q$, such that
\begin{equation*}
a_{i,j} > a_{k,l} \implica x_i + y_j > x_k + y_l,
\end{equation*}
for all $i$, $k\in \nume p$ and all $j$, $l \in \nume q$.
Let $\alpha$, $\beta$ be partitions of the same size.
Denote by $\dmata$ the set of matrices $A$ with nonnegative integer entries,
row-sum vector $\alpha$ and column-sum vector $\beta$.
Also denote by $\parti A$ the sequence of entries of $A$ arranged in weakly
decreasing order.
Our main result is

\begin{mteor} \label{teor:m}
Let $\alpha$, $\beta$, $\gamma$ be partitions of the same size.
If there is an additive matrix $A\in \dmatli \alpha \beta$ with $\parti A = \gamma$,
then, for any triple of partitions $\la$, $\mu$, $\nu$ of the same size, the sequence
$\{\coefili {\la + n \alpha}{\mu + n \beta}{\nu + n \gamma}\}_{n\in \noneg}$
is weakly increasing and bounded from above.
In particular, there is a positive integer $L$ such that, for all $n \ge L$,
\begin{equation*}
\coefili {\la + n \alpha}{\mu + n \beta}{\nu + n \gamma} =
\coefili {\la + L \alpha}{\mu + L \beta}{\nu + L \gamma}.
\end{equation*}
\end{mteor}

The next three results are new instances of stability.
The first two follow from the additivity of matrices $B$ and $C$
in Example~\ref{ejems:adi-tres}; the third from the additivity of the matrix $A$
in Example~\ref{ejem:young}.

\begin{coro} \label{coro:estab-tripode}
Let $a$, $b$, $c$ be nonnegative integers.
Then, for any triple of partitions $\la$, $\mu$, $\nu$ of the same size,
there is an $L\in \natural$, such that for all $n\ge L$,
\begin{multline*}
\coefili {\la + n(b+c+1, 1^a)}{\mu + n(a+c+1, 1^b)}{\nu + n(c+1, 1^{a+b})} \\
= \coefili {\la + L(b+c+1, 1^a)}{\mu + L(a+c+1, 1^b)}{\nu + L(c+1, 1^{a+b})}.
\end{multline*}
\end{coro}

\begin{coro} \label{coro:estab-piramide}
For any $k\in \natural$, let $\eta^k = (\binom{k+1}{2}, \dots, \binom{2}{2})$
and $\delta^k = (k, (k-1)^2, \dots, 1^k)$.
Then, for any triple of partitions $\la$, $\mu$, $\nu$ of the same size,
there is an $L\in \natural$, such that for all $n\ge L$,
\begin{equation*}
\coefili {\la + n \eta^k} {\mu + n \eta^k} {\nu + n \delta^k} =
\coefili {\la + L \eta^k} {\mu + L \eta^k} {\nu + L \delta^k}.
\end{equation*}
\end{coro}

\begin{coro} \label{coro:adi-young}
Let $\alpha$ be a partition.
Then, for any triple of partitions $\la$, $\mu$, $\nu$ of the same size,
there is an $L\in \natural$, such that for all $n\ge L$,
\begin{equation*}
\coefili {\la + n\alpha}{\beta + n \alpha^\prime}{\nu + (n^{|\alpha|})} =
\coefili {\la + L\alpha}{\beta + L \alpha^\prime}{\nu + (L^{|\alpha|})}.
\footnote{This Corollary also follows from Example~6.3(b) in~\cite{stem1}.}
\end{equation*}
\end{coro}

Stembridge's hypothesis~\eqref{ecua:stem} and additivity are closely related.
In Section~\ref{sec:equiv} we show that additivity implies condition~\eqref{ecua:stem},
and that this condition implies the existence of certain unique additive matrix.
The advantage of our hypothesis is that it is far easier to find
examples of additive matrices, than to find examples of partitions
$\alpha$, $\beta$, $\gamma$ satisfying~\eqref{ecua:stem}
(even with help of identity~\eqref{ecua:permu-mat}),
without knowing the relation between them.
Therefore, it is much easier to find new examples of stability.
Instances this are given in Corollaries~\ref{coro:estab-tripode} and~\ref{coro:estab-piramide}.
Moreover, to decide if a matrix is additive can be done in polynomial
time~\cite{onva} (see also Paragraph~\ref{para:comple}).
Another advantage is that there is already a body of results concerning
additive matrices~\cite{dupe, flrs, fishe, onva, sava, vapir, aditivo, elenot, libro}.

The paper is organized as follows.
Section~\ref{sec:cara} contains notation and several known results from
the representation theory of the symmetric group needed in this work.
Section~\ref{sec:disctom} contains a summary of all results about uniqueness
and additivity required for this paper, as well as examples and one application
to Kronecker coefficients needed in the proof of additive stability.
In Section~\ref{sec:geomadi} we present a geometric characterization of
additivity from~\cite{onva}.
This will be fundamental in the proof of the main theorem and also in
Section~\ref{sec:equiv}, where we study the relation between additivity
and condition~\eqref{ecua:stem}.
In this section we also present a new characterization of additivity
(Theorem~\ref{teor:adi-n-uni}).
Section~\ref{sec:prueba} contains the proof of Theorem~\ref{teor:m}.
Finally, in section~\ref{sec:aplica}, we relate additive stability
to the previous known cases of stability of Kronecker coefficients.

\section{Representation theory of the symmetric group} \label{sec:cara} \hfill

We assume the reader is familiar with the standard results in the
representation theory of the symmetric group (see for example~\cite{jake, mac, sag, stan}).
In this section we review some notation, definitions and results used in this paper.

We will use the following notation:
$\natural$ is the set of positive integers, $\noneg =\natural \cup \{0\}$, and,
for any $n\in \noneg$,  $\nume n = \{1, \dots, n\}$, so that $\nume 0 = \varnothing$.
If $\la = \vector \la p$ is a partition (a weakly decreasing sequence of positive integers),
we denote its \emph{size} $\sum_{i\in \nume p} \la_i$ by
$|\la |$ and its \emph{length} $p$ by $\longi{\la}$.
If $|\la|=n$, we also write $\la\vdash n$.
The notation $\la = (1^{a_1}, 2^{a_2}, \dots )$ means that $\la$ has $a_i$ parts
equal to $i$.
We denote by $\la^\prime$ the partition \emph{conjugate} to $\la$, obtained by transposing
the diagram of $\la$.
The {\em depth} of $\la$ is $\prof\la = |\la| - \la_1$.
If $\la$ and $\mu$ are two partitions of length $p$ and $q$, respectively,
and $p \le q$, we denote $\la_{p+1} = \cdots = \la_q = 0$ and define its \emph{sum} by
\begin{equation*}
\la + \mu = (\la_1 + \mu_1, \dots, \la_q + \mu_q).
\end{equation*}

Given two partitions $\la$, $\mu$ of $n$ we write $\la \mayori \mu$
to indicate that $\la$ is greater than or equal to $\mu$ in the dominance
order of partitions.
We write $\la \mayore \mu$, if $\la \mayori \mu$ and $\la \neq \mu$.

The number of semistandard Young tableaux of shape $\la$ and
content $\nu$ is denoted by $\kos\la\nu$.
It is well-known and not difficult to prove that
\begin{equation} \label{ecua:kos-pos}
\kos\la\nu >0  \sii \la \mayori \nu.
\end{equation}

For any partition $\la \vdash n$, we denote by $\cara\la$ the irreducible character
of $\sime n$ associated to $\la$, and, for any partition $\nu = \vector \nu r$ of $n$,
by $\permu\nu = {\rm Ind}_{\sime \nu}^{\sime n}( 1_{\nu})$ the
permutation character associated to $\nu$.
That is, $\permu\nu$ is the character induced from the trivial
character of the Young subgroup
$\sime \nu = \sime{\nu_1} \times \cdots \times \sime{\nu_r}$ to $\sime n$.

The sets $\{ \cara\la \mid \la \vdash n \}$ and $\{ \permu\nu \mid \nu \vdash n \}$
are additive basis of the character ring of $\sime n$.
Both basis are related by Young's rule

\begin{equation} \label{ecua:young}
\permu\nu = \sum_{\la \mayori \nu} \kos\la\nu \cara\la.
\end{equation}

The following identity is well-known and will be very useful in this paper.
It can be found in~\cite[Prop.~2]{dfr} or in~\cite[2.9.16]{jake}.

\begin{equation} \label{ecua:permu-tensor}
\permu\la \otimes \permu\mu = \sum_{A \in \dmat} \permu{\parti A}.
\end{equation}

Due to the well-known identity
\begin{equation*}
\coefi = \frac{1}{n!} \sum_{\sigma\in \sime n} \cara\la(\sigma) \cara\mu(\sigma) \cara\nu(\sigma),
\end{equation*}
the Kronecker coefficient $\coefi$ is symmetric in $\la$, $\mu$ and $\nu$.
We will use this property in Section~\ref{sec:aplica}.

Combining identities~\eqref{ecua:young} and~\eqref{ecua:permu-tensor}, we get

\begin{equation} \label{ecua:permu-mat}
\bili{\permu \la \otimes \permu \mu}{\cara \nu} =
\sum_{A \in \dmat} K_{\nu, \parti A}.
\end{equation}

The next theorem is due to Manivel.

\begin{teor} \label{teor:kron-mani}
{\em \cite[p.~157]{mani}}
Let $\alpha$, $\beta$, $\gamma$, $\la$, $\mu$, $\nu$ be partitions such
that $|\alpha| = |\beta| = |\gamma|$ and $|\la| = |\mu| = |\nu|$.
If $\coefili \alpha\beta\gamma$ and $\coefi$ are non-zero, then
\begin{equation*}
\coefili{\la + \alpha}{\mu + \beta}{\nu + \gamma} \ge \max\{ \coefili \alpha\beta\gamma,
\coefi \}.
\end{equation*}
\end{teor}

\section{Discrete tomography} \label{sec:disctom}
\hfill

In this section we gather several results from discrete tomography that
are used to prove additive stability.
Applications of discrete tomography to Kronecker coefficients have
already been given before in~\cite{vapp, avva}.
Two notions from discrete tomography are relevant for this paper: uniqueness
and additivity of 3-dimensional binary matrices (see~\cite{fishe, libro}).
They are related to the notions of minimality, $\pi$-uniqueness~\cite{tova,elenot}
and additivity~\cite{aditivo, onva, elenot}, which are defined for matrices with nonnegative
integer entries.
Below we explain these concepts, the relations among them, give examples
and show how they are related to Kronecker coefficients.

\begin{intmat}
Let $A=(a_{i,j})$ be a matrix of size $p \times q$.
We associate to $A$ two compositions and one partition.
For each $i\in \nume p$ and each $j\in \nume q$, one defines
\begin{equation*}
\alpha_i = \sum_{y \in \nume q} a_{i,y} \quad \text{and} \quad
\beta_j = \sum_{x \in \nume p} a_{x,j}.
\end{equation*}
Then, $\alpha = \vector \alpha p$ is called the \emph{row-sum} vector and
$\beta = \vector \beta q$ the \emph{column-sum} vector of $A$.
We denote by $\parti A$ the vector of entries of $A$ arranged in weakly
decreasing order.
Then, $\parti A$ is a partition called the \emph{$\pi$-sequence} of $A$.
We say that $A$ is a \emph{plane partition} if it has nonnegative integer
entries and its rows and columns are weakly decreasing.
For the applications we have in mind we assume, \textbf{from now on},
without loss of generality, that the row-sum and column-sum vectors are
weakly decreasing.
Otherwise we just permute rows and columns.

We denote by $\dmata$ the set of all matrices $A=(a_{i,j})$ with nonnegative
integer entries, row-sum vector $\alpha$ and column-sum vector $\beta$.
If $\gamma$ is a partition, we denote by $\dmatapi$ the set of all matrices
in $\dmata$ with $\pi$-sequence $\gamma$.
\end{intmat}

\begin{binmat}
Let $A=(a_{i,j,k})$ be a matrix of size $p \times q \times r$ with
entries in $\{0, 1\}$ (we call it a binary matrix for short).
We associate to $A$ three compositions.
For each $i \in \nume p$, $j\in \nume q$ and $k\in \nume r$, one defines
\begin{equation*}
\alpha_i = \sum_{(y,z) \in \nume q \times \nume r} a_{i,y,z},\quad
\beta_j = \sum_{(x,z) \in \nume p \times \nume r} a_{x,j,z} \quad
\textrm{and} \quad
\gamma_k = \sum_{(x,y)\in \nume p \times \nume q} a_{x,y,k}.
\end{equation*}
Then, the compositions $\alpha = \vector \alpha p$, $\beta = \vector \beta q$,
$\gamma = \vector \gamma r$ are called the \emph{1-marginals} of $A$.
See, for example~\cite{delki}, where 3-dimensional binary matrices
are called three-way statistical tables.
For the applications we have in mind we assume, \textbf{from now on},
without loss of generality, that the 1-marginals are weakly decreasing.
Otherwise we just permute 2-dimensional slices.

We denote by $\tmata$ the set of all 3-dimensional binary matrices
with 1-marginals $\alpha$, $\beta$, $\gamma$,
and by $\tmatanum$ its cardinality.
A matrix $X\in \tmata$ is called a \emph{matrix of uniqueness}~\cite{flrs},
if $\tmatanum = 1$.
\end{binmat}

\begin{morintmat}
Let $\alpha$, $\beta$ be partitions of the same size and let
$A = (a_{i,j}) \in \dmata$.
We say that $A$ is \emph{minimal}~\cite{tova} if there is no matrix $B\in \dmata$
such that $\parti B \menore \parti A$, and
we say that $A$ is \emph{$\pi$-unique} if there is no other matrix $B\in \dmata$
such that $\parti B = \parti A$.
Suppose $A$ has size $p\times q$, and let $r$ be the maximum of the entries of $A$.
The \emph{graph} of $A$ is the 3-dimensional binary matrix
$\grafi A = (a_{i,j,k})$ defined, for all
$(i,j,k) \in \nume p \times \nume q \times \nume r$, by
\begin{equation*}
a_{i,j,k} =
\begin{cases}
1 & \text{if $1 \le k \le a_{i,j}$;} \\
0 & \text{otherwise.}
\end{cases}
\end{equation*}
Note that, if $A\in \dmatapi$, then $\grafi A \in \tmatac$.
Therefore,
\begin{equation*}
\fun {\sf G} {\dmatapi}{\tmatac}
\end{equation*}
is a well-defined injective map.
If $X$ is the image $\grafi A$ of a plane partition $A$, then $X$ is called
the \emph{diagram} of $A$~\cite{stanspp} or \emph{pyramid}~\cite{vapir}.
\end{morintmat}

The next theorem relates the property of uniqueness for binary matrices
to properties of integral matrices.

\begin{teor} \label{teor:tova}
\emph{\cite[Thm.~1]{tova}}
Let $\alpha$, $\beta$, $\gamma$ be partitions.
Then $\tmatacnum = 1$ if and only if there is a matrix $A\in \dmatapi$ that is
minimal and $\pi$-unique.
Moreover, if $A \in \dmata$ is minimal and $\pi$-unique, then $A$ is a plane partition.
\end{teor}

The theorem can be rephrased in the following way:
Let $A\in \dmatapi$.
Then, $\grafi A$ is a matrix of uniqueness if and only if $A$ is minimal
and $\pi$-unique.
Moreover, if $\grafi A$ is a matrix of uniqueness, then $A$ is a plane partition.

To the best of our knowledge, the first application of uniqueness to
Kronecker coefficients is Theorem~1.1 in~\cite{vapp}.
Here, we need only a particular case.

\begin{teor} \label{teor:uni-kron}
{\em \cite[Cor.~4.2]{vapp}}
Let $A\in \dmatapi$ be a matrix that is minimal and $\pi$-unique.
Then,\footnote{Compare Theorems~\ref{teor:tova} and~\ref{teor:uni-kron} here with
Theorem~6.4 in~\cite{stem1}.}
\begin{equation*}
\coefili \alpha \beta\gamma = 1.
\end{equation*}
\end{teor}

We have the following characterization of minimality.

\begin{teor} \label{teor:cara-mini}
\emph{\cite[Prop.~3.1]{vapp}}
Let $A\in \dmatapi$.
Then, $A$ is a minimal matrix if and only if the graph map
\begin{equation*}
\fun {\sf G} {\dmatapi}{\tmatac}
\end{equation*}
is bijective.
\end{teor}

The proof in~\cite{vapp} uses characters.
For a combinatorial proof see either~\cite[Thm.~5.3]{elenot}
or~\cite[Thm.~13]{libro}.
The classification of minimal matrices of size $2 \times q$
is known~\cite[Thm.~1.1]{minidosq}, but, in general, it is not easy
to decide if a matrix is minimal and $\pi$-unique (see~\cite[Thm.~2.7]{gridev}.

Now we turn to the notion of additivity, which is fundamental for our
main theorem.
It appears already, with no name, in the characterization
of $(0,1)$-additivity given in~\cite{aditivo}.

\begin{defi}
Let $A=(a_{i,j})$ be a matrix of size $p\times q$ with nonnegative
integer entries.
Then, $A$ is called \emph{additive}~\cite[\S~6]{elenot} if there exist real numbers
$x_1, \dots, x_p$, $y_1, \dots, y_q$ such that
\begin{equation*}
a_{i,j} > a_{k,l} \implica x_i + y_j > x_k + y_l,
\end{equation*}
for all $i$, $k\in \nume p$ and all $j$, $l\in \nume q$.
Let $X = (x_{i,j,k})$ be a binary matrix of size $p \times q \times r$.
$X$ is called \emph{$(0,1)$-additive}~\cite[p.~150]{flrs} if there are real numbers
$x_1, \dots, x_p$, $y_1, \dots, y_q$, $z_1, \dots, z_r$ such that,
for all $i\in \nume p$, $j\in \nume q$, $k\in \nume r$,
\begin{equation} \label{ecua:adi-bin}
x_{i,j,k} = 1 \sii x_i + y_j + z_k \ge 0.
\end{equation}
Note that $X$ is called additive in~\cite{flrs}.
We call it $(0,1)$-additive to distinguish it from the other concept of
additivity.
\end{defi}

First we show how these two concepts are related.

\begin{teor} \label{teor:adi-adi}
{\em \cite[Thm.~1]{aditivo}}
Let $A$ be a matrix with nonnegative integer entries.
Then $A$ is additive if and only if $\grafi A$ is $(0,1)$-additive.
\end{teor}

In the next example we illustrate the concepts of minimal, $\pi$-unique
and additive (see Example~8 in~\cite{libro}).

\begin{ejem} \label{ejem:mua}
Let
\begin{equation*}
A=
\begin{bmatrix}
3 & 3 & 1 \\
2 & 1 & 1 \\
2 & 0 & 0
\end{bmatrix}
,
\quad
B=
\begin{bmatrix}
4 & 4 & 1 \\
2 & 1 & 1 \\
2 & 0 & 0
\end{bmatrix}
\quad \textrm{and} \quad
C=
\begin{bmatrix}
4 & 3 & 2 \\
3 & 1 & 0 \\
1 & 1 & 0
\end{bmatrix}
.
\end{equation*}
The three matrices $A$, $B$, $C$ are plane partitions.
The first one is minimal, but not $\pi$-unique.
The first assertion can be checked directly by hand;
for the second just take the transpose of $A$.
The matrix $B$ is $\pi$-unique, but not minimal.
The first assertion can be checked by hand; for the second observe
that $C$ has the same 1-marginals as $B$, and
$\parti C \menore \parti B$.
Finally the matrix $C$ is additive.
To see this take $(x_1, x_2, x_3) = (7,2,0)$ and $(y_1, y_2, y_3) = (6,3,0)$.
\end{ejem}

Now we show some relations between being additive, minimal, $\pi$-unique
and a plane partition.
We start with the following result, which is an equivalent formulation of~\cite[Thm.~2]{aditivo}.
See~\cite[Thm.~15]{libro} for a simpler proof.

\begin{teor} \label{teor:dpq}
{\em \cite[Thm.~2]{aditivo}}
Let $A \in \dmata$ of size $2 \times q$.
Then, the conditions of $A$ being a plane partition; minimal and
$\pi$-unique; and additive are equivalent.\footnote{Compare with Proposition~6.9
in~\cite{stem1}.}
\end{teor}

This result does not hold in general.
Already, the matrix $A$ in Example~\ref{ejem:mua} is a plane partition
that is not $\pi$-unique, and $B$ is a plane partition that is not
minimal.
Obstructions for a plane partition $A$ of size $3\times 3$ to be minimal
and $\pi$-unique are given in~\cite[\S5]{vapir}.
These correspond to the obstructions for additivity given by the arrow diagram
in~\cite[Fig.~3]{sava} and its transpose.
It is not difficult to show that these are all obstructions to additivity
for a plane partition of size $3\times 3$ (see~\cite[Ch.~2]{san}).\footnote{Compare
with Proposition~6.11 in~\cite{stem1}.}

We have however one implication (see also the proof of Corollary~3.4 in~\cite{onva}).

\begin{teor} \label{teor:a-mu}
\emph{\cite[Thm.~6.1]{aditivo}}
Let $A$ be an additive matrix, then $A$ is minimal and $\pi$-unique.
\end{teor}

The converse is not true (see~\cite[Ex.~7.4]{elenot} and~\cite[\S3]{sava}).

\begin{ejem} \label{ejem:uni-no-adi}
The matrix
\begin{equation*}
A =
\begin{bmatrix}
5 & 5 & 5 & 4 & 4 \\
5 & 5 & 5 & 3 & 3 \\
3 & 3 & 1 & 1 & 0 \\
2 & 1 & 1 & 1 & 0 \\
2 & 1 & 0 & 0 & 0
\end{bmatrix}
\end{equation*}
is minimal and $\pi$-unique but not additive.
\end{ejem}

The next result follows from Theorems~\ref{teor:a-mu} and~\ref{teor:tova}.
Recall that we are assuming that the row-sum and column-sum vectors are
weakly decreasing.
For a simple direct proof see~\cite[Lemma~2.6]{sava}

\begin{coro}
Every additive matrix is a plane partition.
\end{coro}

Next result follows from Theorems~\ref{teor:a-mu} and~\ref{teor:uni-kron}.

\begin{teor} \label{teor:adi-uni-kron}
Let $A \in \dmatapi$ be additive.
Then,
\begin{equation*}
\coefili \alpha\beta\gamma = 1.
\end{equation*}
\end{teor}

Since the application of our main theorem depends on examples of additive
matrices, we will give below some examples.

\begin{ejems} \label{ejems:adi-tres}
The following are additive matrices
\begin{equation*}
A= \left[
\begin{matrix}
r & \cdots & r \\
\vdots & \ddots & \vdots \\
r & \cdots & r
\end{matrix}
\right],
\quad
B= \left[
\begin{matrix}
c+1 & 1 & \cdots & 1 \\
1 & 0 & \cdots & 0 \\
\vdots & \vdots & \ddots & \vdots \\
1 & 0 & \cdots & 0
\end{matrix}
\right],
\quad
C = \left[
\begin{matrix}
k   & k-1 & \cdots & 2 & 1 \\
k-1 & k-2 & \cdots & 1 & 0 \\
\vdots & \vdots &  \ddots & \vdots & \vdots \\
2 & 1 & \cdots & 0 & 0 \\
1 & 0 & \cdots & 0 & 0
\end{matrix}
\right].
\end{equation*}
The graph of $A$ is a 3-dimensional box of size $p \times q \times r$;
the graph of $B$ is called \emph{tripod}, it has size $(a+1) \times (b+1) \times (c+1)$
and 1-marginals $(b+c+1, 1^a)$, $(a+c+1, 1^b)$, and $(a+b+1, 1^c)$.
In both examples additivity is easy to prove.
The graph of $C$ is a pyramid of size $k\times k \times k$.
To show that $C$ is additive, take $\vector xk = \vector yk = (k-1, k-2, \dots, 1,0)$,
(see~\cite[Ex.~12]{libro}).
\end{ejems}

\begin{ejem} \label{ejem:young}
Let $\alpha$ be a partition.
Let $A$ be the only binary matrix in $\dmatli \alpha{\alpha^\prime}$.
Then $A$ is additive.
For example, if $\alpha =(4,2,1)$, then
\begin{equation*}
A =
\begin{bmatrix}
1 & 1 & 1 & 1 \\
1 & 1 & 0 & 0 \\
1 & 0 & 0 & 0
\end{bmatrix}.
\end{equation*}
To see that $A$ is additive take $x_i = \alpha_i$ and $y_j = (\alpha^\prime)_j$.
\end{ejem}

\begin{simetria} \label{para:sime}
The natural action of the symmetric group $\sime 3$ in $\natural^3$ permuting coordinates
induces an action of $\sime 3$ on the set of finite subsets of $\natural^3$, and therefore
on 3-dimensional binary matrices.
Since the definition of $(0,1)$-additivity is symmetric under the action of $\sime 3$
(see~\eqref{ecua:adi-bin}), we have, for any $\sigma\in \sime 3$, that
$X$ is $(0,1)$-additive if and only if $ \sigma X$ is $(0,1)$-additive.
The action of $\sime 3$ on $\natural^3$ restricts to the set of pyramids, and
therefore to the set of plane partitions.
Thus, because of Theorem~\ref{teor:adi-adi},
for any plane partition $A$ and any $\sigma\in \sime 3$, we have that
$A$ is additive if and only if $\sigma A$ is additive.
There is also another operation on plane partitions called complementation~\cite[\S~2]{stanspp},
which combined with the elements in $\sime 3$ generates a group $T$ with 12 elements.
It is not difficult to prove for a plane partition $A$ that, $A$ is additive if and only if
its complement is additive.
Therefore if $A$ is an additive plane partition, we can generate with the action of $T$
up to 12 different additive matrices.
For example, since the matrix $C$ in Example~\ref{ejem:mua} is additive, then
the matrices
\begin{equation*}
\begin{bmatrix}
3 & 2 & 2 \\
3 & 1 & 0 \\
2 & 1 & 0 \\
1 & 0 & 0
\end{bmatrix},
\text{ and }
\begin{bmatrix}
4 & 3 & 3 \\
4 & 3 & 1 \\
2 & 1 & 0
\end{bmatrix}
\end{equation*}
are additive.
The first is $(1\, 2\, 3)C$, where $(1\, 2\, 3)$ is a 3-cycle in $\sime 3$;
the second is the complement of $C$.
One can also prove similar results for matrices of uniqueness, but we will not need
them here.
\end{simetria}

\begin{defi} \label{defi:adi-tri}
We say that $(\alpha, \beta, \gamma)$ is an \emph{additive triple} if
there is an additive matrix $A\in\dmatapi$.
Using the $\sime 3$-action on $\grafi A$ we obtain that
$(\gamma^\prime, \alpha, \beta^\prime)$,
$(\beta,\gamma^\prime,\alpha^\prime)$,
$(\beta,\alpha,\gamma)$,
$(\alpha,\gamma^\prime,\beta^\prime)$ and
$(\gamma^\prime,\beta,\alpha^\prime)$
are also additive triples.
\end{defi}

\begin{obse}
The additive triples $(\alpha, \beta, \gamma)$ and
$(\gamma^\prime, \alpha, \beta^\prime)$ yield different stability
properties.
The identity $\coefi = \coefili \la{\mu^\prime}{\nu^\prime}$ and
other symmetries of Kronecker coefficients are not enough to prove that
they yield  the same stability property, because, in general
$(\mu + \beta)^\prime \neq \mu^\prime + \beta^\prime$.
\end{obse}

\begin{ejem} \label{ejem:adi-tri}
Let $\beta = \vector \beta b\vdash n$.
Then,
$ B =
\begin{bmatrix}
\beta_1 & \cdots & \beta_b
\end{bmatrix}$
is an additive matrix and $((n), \beta, \beta)$ is an additive triple.
Hence, also $(\beta, \beta^\prime, (1^n)) $
is an additive triple.
This is Example~\ref{ejem:young}.
\end{ejem}

\begin{obstru}
There are also ways of proving that a matrix is not additive.
In~\cite[Thm.~3.8]{sava} we showed that certain arrow diagrams are
obstructions to additivity.
While the case $2\times q$ is fairly simple (Theorem~\ref{teor:dpq}),
the general case is much more complex.
We showed in~\cite[\S~5]{sava} that there are infinitely many
essentially different obstructions needed for deciding additivity
of plane partitions with three rows.
\end{obstru}

\begin{complejidad} \label{para:comple}
Despite the existence of infinitely many essentially different obstructions
for deciding additivity of a matrix with nonnegative integer entries,
this can be done in polynomial time~\cite[Thm.~7.1]{onva}.
In contrast, deciding uniqueness or if $\tmatanum$ is positive are each NP-complete
(see~\cite[Thm.~3.1]{bdg} and~\cite[Thm.~2.7]{gridev}).
\end{complejidad}

\section{Geometry of additive matrices} \label{sec:geomadi}
\hfill

In this section we record a geometric characterizations of minimality
and additivity from~\cite{onva}.
One of them will be central in our proof of additive stability.

We start by extending some notions defined for objects with integer
entries to objects with real entries.
For a vector $a=\vector a m\in\real^m$, we denote by
$\parti{a}= (a_{\conjunto 1},\dots, a_{\conjunto m})$
the vector formed by the entries of $a$ arranged in weakly
decreasing order.
We say that $a$ is \emph{majorized} by $b = \vector bm$ (see~\cite{halipo, maol}),
and denote it by $a \menori b$, if
\begin{equation*}
\sum_{i=1}^m a_i=\sum_{i=1}^m b_i,\ \text{and}\
\sum_{1=1}^k a_{\conjunto i} \le \sum_{i=1}^k b_{\conjunto i},
\text{ for all } k \in \nume m.
\end{equation*}
If $a \menori b$ and $\parti a\neq \parti b$,
then we write $a \menore b$.

Let $\alpha$, $\beta$ be two partitions of the same size.
Denote by $\trapo$ the set of all matrices with nonnegative real entries,
row-sum vector $\alpha$ and column-sum vector $\beta$.
$\trapo$ is called a \emph{transportation polytope}.
We say that a matrix $A\in \trapo$ is \emph{real-minimal}~\cite{onva}, if there
is no other matrix $B\in \trapo$ such that $\parti B \menore \parti A$.
Also the definition of additivity can be extended in a straightforward manner
to matrices with real entries.

\begin{teor} \label{teor:adi-rm}
{\em \cite[Thm.~6.2]{onva}}
Let $A\in\trapo$.
Then $A$ is additive if and only if $A$ is real-minimal.
\end{teor}

Let $a\in\real^m$ and $\rho$ be a permutation in the symmetric
group $\sime m$.
Denote by $a_\rho$ the vector $ (a_{\rho(1)},\dots,a_{\rho(m)})$.
The \emph{permutohedron} determined by $a$ is the convex hull
of the set of all vectors obtained by permuting the entries of $a$:
\begin{equation*}
\permuta a={\rm conv}
\{ a_\rho \mid \rho \in {\sf S}_m \} .
\end{equation*}
It is a convex polytope whose set of vertices is precisely
$\{ a_\rho \mid \rho \in {\sf S}_m \}$.
More generally, its face lattice is known; see for example \cite{bisa, ykk}.

We will make use of the following theorem of Rado (see also \cite[p.~113]{maol})

\begin{teor} \label{teor:rado}
\emph{\cite{ra}}
For any vector $a\in\real^m$
\begin{equation*}
\permuta a=\{ x\in\real^m\mid x \menori a \}.
\end{equation*}
\end{teor}

In order to state our next results, we denote, by ${\sf M}_{p,q}$ the set
of all matrices with real entries of size $p\times q$ and define
a linear isomorphism $\fun \Phi {{\sf M}_{p,q}} {\real^{pq}}$,
for each $A=(a_{i,j})$, by
\begin{equation*}
\Phi(A) = (a_{11}, a_{12}, \dots, a_{1q}, a_{21}, a_{22}, \dots, a_{2q},
\dots, a_{p1}, a_{p2}, \dots, a_{pq}).
\end{equation*}

We have the following characterization of minimality.

\begin{prop}
{\em \cite[Thm.~5.6]{onva}}
Let $A \in \dmatapi$.
Then $A$ is minimal if and only if
\begin{equation*}
\permuta {\fimap A} \cap \fimap {\transport \alpha \beta} \cap \entero^{pq} =
\{ \fimap B \mid B \in \dmatapi \}.
\end{equation*}
\end{prop}

\begin{ejem}
The matrix $A$ in Example~\ref{ejem:mua}, and its transpose $A^\top$ are minimal.
Then the $\permuta {\fimap A} \cap \fimap {\transport \alpha \beta}$ has
exactly two integer points: $\fimap A$ and $\fimap {A^\top}$.
\end{ejem}

We have the following characterization of real-minimality.

\begin{prop} \label{prop:rm-pit}
{\em \cite[Cor.~5.2]{onva}}
Let $A\in \transport \alpha \beta$.
Then $A$ is real-minimal if and only if
\begin{equation*}
\permuta {\fimap A} \cap \fimap {\transport \alpha \beta} = \{ \fimap A \}.
\end{equation*}
\end{prop}

Next result is fundamental in our proof of additive stability.
It is a consequence of Theorem~\ref{teor:adi-rm}
and Proposition~\ref{prop:rm-pit}.

\begin{teor} \label{teor:adi-pit}
Let $A\in \trapo$.
Then $A$ is additive if and only if
\begin{equation*}
\permuta {\parti A} \cap \fimap {\transport \alpha \beta} = \{ \fimap A \}.
\end{equation*}
\end{teor}

\begin{obse} \label{obse:donde-permuta}
Let $A \in {\sf M}_{p,q}$ have nonnegative entries, and let
$M$ be the sum of the entries of $A$.
In Theorem~\ref{teor:adi-pit} we assume that $\parti A$ has $pq$ coordinates,
by adding zeros, if necessary, so that $\permuta{\parti A}$ is a polytope contained
in the hyperplane $H_M$ of $\real^{pq}$ defined by the equation
$\sum_{i\in \nume{pq}} x_i = M$.
\end{obse}

\begin{ejem} \label{ejem:uni-si-adi-no}
Let
\begin{equation*}
A =
\begin{bmatrix}
5 & 5 & 5 & 4 & 4 \\
5 & 5 & 5 & 3 & 3 \\
3 & 3 & 1 & 1 & 0 \\
2 & 1 & 1 & 1 & 0 \\
2 & 1 & 0 & 0 & 0
\end{bmatrix}
\quad \text{and} \quad
X =
\begin{bmatrix}
0 & 0 & 0 & -1 & 1 \\
0 & 0 & 1 & -1 & 0 \\
0 & 1 & 0 & 0 & -1 \\
-1 & -1 & 0 & 2 & 0 \\
1 & 0 & -1 & 0 & 0
\end{bmatrix}.
\end{equation*}
The matrix $X$ comes from the obstruction to additivity
given in~\cite[Fig.~4]{sava}.
Let $\alpha = (23,21,8,5,3)$, $\beta =(17,15,12,9,7)$.
Since the row-sum and column-sum vectors of $X$ are zero,
$A$ and $A- \frac{1}{2}X$ are elements of $\dmata$.
We know, by Example~\ref{ejem:uni-no-adi}, that $A$ is minimal and
$\pi$-unique.
One easily checks that $\parti{A-X} \menore \parti A$.
Therefore $A$ is not real minimal.
Hence, $\permuta {\parti A} \cap \fimap {\transport \alpha \beta}$
has only one integer point, but it is not
$0$-dimensional.\footnote{This example disproves Conjecture~6.7 in~\cite{stem1}.}
\end{ejem}

\begin{obse}
Propositions~5.8 and~5.9 in~\cite{onva} show how to construct real minimal,
respectively, minimal matrices using quadratic programming.
\end{obse}

\section{Proof of the main theorem} \label{sec:prueba}
\hfill

In this section we prove Theorem~\ref{teor:m}.

\begin{nota}
For each convex polytope $P$ let us denote by $\# P$ the number
of integer points in $P$.
\end{nota}

\begin{prop} \label{prop:pt-acota}
Let $A \in \dmatapi$ be additive.
Then, for any triple of partitions $\la$, $\mu$, $\nu$ of the same size,
the sequence of integers
\begin{equation*}
\left\{ \# \permuta{\nu + n \gamma} \cap
\fimap{\transport {\la + n \alpha}{\mu + n \beta}} \right\}_{n \in \natural}
\end{equation*}
is weakly increasing and bounded from above.
\end{prop}
\begin{proof}
Let us denote $p = \max\{\longi\alpha, \longi \la \}$ and
$q = \max \{ \longi \beta, \longi \mu \}$.
If $\longi{\nu + n \gamma} >pq$, any vector in $\permuta{\nu + n \gamma}$
would have more than $pq$ non-zero coordinates.
Hence, the intersection $\permuta{\nu + n \gamma} \cap
\fimap{\transport {\la + n \alpha}{\mu + n \beta}}$ would be empty.
So, we assume without loss of generality that $\longi{\nu + n \gamma} \le pq$,
and we add, if necessary, zeros at the end of $\nu + n\gamma$, so that
$\permuta{\nu + n \gamma}$ is contained in $\real^{pq}$.

Since $A$ is additive, then by Theorem~\ref{teor:adi-pit},
$\fimapinv{\permuta \gamma} \cap \transport \alpha \beta = \{ A\}$.
For each $n\in \natural_0$, let
\begin{equation*}
R_n = \fimapinv{\permuta{\nu + n \gamma}} \cap \transport {\la + n \alpha}{\mu + n \beta}.
\end{equation*}
If for all $n\in \natural_0$, $R_n = \vacio$, our claim follows trivially.
So, we assume the opposite and denote by $m$ the smallest $n$ such that $R_n \neq \vacio$.
For any $n \ge m$, let $\fun {f_n}{R_n}{R_{n+1}}$ be defined by $f_n(X) = X + A$,
for each $X\in R_n$.
Since $\parti {X+A} \menori \parti X + \parti A$, then, by Theorem~\ref{teor:rado},
$f_n$ is a well-defined injective map, that sends integer points to integer points.
We will show that for $n$ sufficiently large, $f_n$ is bijective.
From this our claim will follow, because the number of integer points in $R_n$ is
at most the cardinality of $\dmatli {\la + n \alpha}{\mu + n \beta}$, which is finite.
Let $B$ be a ball of radius $\frac{1}{3}$ and center $A$ in ${\sf M}_{p,q}$.
Since the sequence of polytopes $\{Q_n\}_{n \ge m}$, where
\begin{equation*}
Q_n = \textstyle
\fimapinv{\permuta{\frac{1}{n} \nu + \gamma}} \cap
\transport {\frac{1}{n}\la + \alpha}{\frac{1}{n} \mu + \beta},
\end{equation*}
converges to $\fimapinv{\permuta \gamma} \cap \transport \alpha \beta$, in the
Hausdorff metric~\cite[p.~279]{mun},
there is an $N\in \natural$ such that for all $n\ge N$, one has $Q_n \subseteq B$.
Let $c(1), \dots , c(pq)$ be the elements of $\nume p \times \nume q$ arranged in such a way
that $\parti A = (a_{c(1)}, \dots, a_{c(pq)})$.
Let
\begin{equation*}
d = 1 + \max\{ a_{c(i)} - a_{c(i+1)} \mid 1 \le i < pq \},
\end{equation*}
and let $n \ge \max\{ N, 3d \}$.
Let $X \in R_n$.
Hence $\frac{1}{n} X$ is in $Q_n$.
Therefore $\left\| \frac{1}{n} X - A \right\| \le \frac{1}{3}$, and we conclude
\begin{equation*}
\textstyle
\left|\frac{1}{n} x_{c(i)} - a_{c(i)} \right| \le \frac{1}{3},
\end{equation*}
for all $i \in \nume{pq}$.
If $a_{c(i)} > a_{c(i+1)}$, then $\frac{1}{n} x_{c(i)} - \frac{1}{n} x_{c(i+1)} \ge \frac{1}{3}$.
So, we get
\begin{equation*} \textstyle
x_{c(i)} - x_{c(i+1)} \ge \frac{n}{3} \ge d > a_{c(i)} - a_{c(i+1)} > 0.
\end{equation*}
From this we get $x_{c(i)} > x_{c(i+1)}$, as well, $x_{c(i)} - a_{c(i)} > x_{c(i+1)} - a_{c(i+1)}$.
If $a_{c(i)} = a_{c(i+1)}$, we can relabel $c(i)$ and $c(i+1)$ so as to get
$x_{c(i)} \ge x_{c(i+1)}$ and $x_{c(i)} - a_{c(i)} \ge x_{c(i+1)} - a_{c(i+1)}$.
In both cases we get
\begin{equation*}
\parti X = (x_{c(1)}, \dots, x_{c(pq)}) \text{ and }
\parti{X - A} = \parti X - \parti A.
\end{equation*}
Thus, $\parti{X - A}  \menori \nu + n\gamma - \gamma$.
In particular, $x_{c(i)} - a_{c(i)} \ge 0$, for all $i \in \nume{pq}$.
It follows that $X-A$ is in $R_{n-1}$ and therefore $f_{n-1}$ is bijective.
The proposition is proved.
\end{proof}

\begin{obse} \label{obse:pt-acota}
Let $E_n = \fimapinv{\permuta{\nu + n \gamma}} \cap \dmatli {\la + n \alpha}{\mu + n \beta}$.
It follows from the proof of the previous proposition that, there is an $m\in \natural$,
such that, for all $n\ge m$, the maps $\fun{g_n}{E_n}{E_{n+1}}$, defined by $g_n(X) = X+A$,
are bijective.
Hence, if $E_m= \{ X_1, \dots, X_l \}$, then $E_{m + k} = \{ X_1 + kA, \dots, X_l + kA \}$.
Moreover, for each $j\in \nume l$, there is an ordering
$c^j(1), \dots, c^j(pq)$ of the elements of $\nume p \times \nume q$ such that,
$\parti A = (a_{c^j(1)}, \dots, a_{c^j(pq)}) = \gamma$,
$\parti{X_j} = (x_{j,c^j(1)}, \dots, x_{j,c^j(pq)})$ and
$\parti{X_j + kA} = \parti{X_j} + k \gamma$.
Note that two orderings $c^i$, $c^j$ may differ only at numbers $s$, $t$
such that $a_{c^i(s)} = a_{c^i(t)}$ and $a_{c^j(s)} = a_{c^j(t)}$.
\end{obse}

The following construction is equivalent to the Gelfand-Tsetlin patterns
appearing in~\cite[p.~313]{stan}, but it is more suited for our purposes.
It follows the same idea from~\cite[\S3]{pava}.

\begin{getepat}
Let $\sigma$, $\gamma \in \noneg^\ell$ be vectors such that, $\sigma$ is weakly
decreasing and $\sum_{i\in \nume \ell} \sigma_i = \sum_{i\in \nume \ell} \gamma_i$.
We denote by $\gtpat \sigma \gamma$ the set of real triangular arrays
$X= \gtele xij$ with $1 \le i \le j \le \ell$, such that
the following conditions are satisfied:

(Po) $x_{i,j} \ge 0$, for all $1 \le i \le j \le \ell$;

(CS) $\sum_{j = i}^m x_{i,j}  \ge \sum_{j = i+1}^{m+1} x_{i+1,j}$,
for all $1\le i \le m < \ell$.

(Sh) $\sum_{j = i}^\ell x_{i,j} = \sigma_i$, for all $i\in \nume \ell$.

(Co) $\sum_{i=1}^j x_{i,j} = \gamma_j$, for all $j \in \nume \ell$.

\noindent
The \emph{shape} of such an array is $\sigma$ and the \emph{content} is $\gamma$.

The triangles with integer coordinates in $\gtpat \sigma \gamma$ correspond to
semistandard Young tableaux  of shape $\sigma$ and content $\gamma$ in the following way:
if $T$ is such a tableau, we define $x_{i,j}$ as the number of $j$'s in the $i$-th row
of $T$.
Then, condition (CS) is equivalent to $T$ being column-strict.
$\gtpat \sigma \gamma$ is a convex polytope and
\begin{equation*}
\# \gtpat \sigma \gamma = K_{\sigma\gamma}.
\end{equation*}
We denote by $C_\sigma = \gtele cij$ the triangular array defined by $c_{i,i} = \sigma_i$
for all $i \in \nume \ell$ and $c_{i,j} = 0$ if $1\le i < j \le \ell$.
Then, $C_\sigma$ is the only element in $\gtpat \sigma \sigma$.
\end{getepat}

\begin{prop} \label{prop:k-acota}
Let $\gamma$ and $\nu$ be a partitions and $\rho$ be a vector with nonnegative
integer entries, whose coordinates add up to $|\nu|$.
Then, the sequence $\left\{ K_{\nu + n \gamma, \rho + n \gamma} \right\}_{n \in \natural}$ is
weakly increasing and bounded from above.\footnote{See Proposition~5.2 in~\cite{stem1}
for a more general result.}
\end{prop}
\begin{proof}
If for all $n\in \noneg$, $\gtpat{\nu + n \gamma}{\rho + n \gamma}$ is empty,
our claim follows trivially.
So, we assume the opposite and denote by $m$ the smallest $n$ such that
$\gtpat{\nu + n \gamma}{\rho + n \gamma}$ is non-empty.
For any $n \ge m$, we define a map
\begin{equation*}
\fun {f_n}{\gtpat{\nu + n \gamma}{\rho + n \gamma}}{\gtpat{\nu + (n + 1) \gamma}{\rho + (n + 1) \gamma}},
\end{equation*}
by $f_n(X) = X + C_\gamma$, for each $X \in \gtpat{\nu + n \gamma}{\rho + n \gamma}$.
It is straightforward to check that $f_n$ is a well-defined injective map that sends
integer points to integer points.
We will show that for $n$ sufficiently large, $f_n$ is bijective.
From this our claim will follow, because the number of integer points in
$\gtpat{\nu + n \gamma}{\rho + n \gamma}$ is finite.
Let $\ell = \max\{ \longi\gamma, \longi \nu \}$.
Let $B$ be a ball of radius $\frac{1}{3\ell}$ and center $C_\gamma$.
Since the sequence of polytopes $\{G_n\}_{n\ge m}$, where
\begin{equation*}
\textstyle
G_n = \gtpat{\frac{1}{n}\nu + \gamma}{\frac{1}{n} \rho + \gamma},
\end{equation*}
converges, in the Hausdorff metric, to $\gtpat \gamma\gamma = \{ C_\gamma \}$,
there is a number $N$, such that for all $n \ge N$, we have
$G_n \subseteq B$.
Let $n > N$ and let $X = \gtele xij$ be an element of
$\gtpat{\nu + (n+1) \gamma}{\rho + (n+1)\gamma}$.
We claim that $X - C_\gamma$ is in $\gtpat{\nu + n \gamma}{\rho + n \gamma}$.
Note that the conditions (Sh) and (Co) are straightforward.
Since $\frac{1}{n+1} X$ is in $G_{n+1} \subseteq B$, we have
$\left| \frac{1}{n+1} x_{i,i} - \gamma_i \right| \le \frac{1}{3\ell}$,
for all $i\in \nume \ell$, and $0 \le \frac{1}{n+1} x_{i,j} \le \frac{1}{3\ell}$,
for all $1 \le i < j \le \ell$.
Combining these inequalities, we get
\begin{equation*}
\gamma_i - \frac{1}{3} \le \frac{1}{n+1} \sum_{j=i}^k x_{i,j}
\le \gamma_i + \frac{1}{3},
\end{equation*}
for all $1 \le i \le k \le \ell$.
Multiplying by $n+1$ and substracting $\gamma_i$, we have
\begin{equation} \label{ecua:prueba-kos}
n \gamma_i - \frac{n+1}{3} \le \sum_{j=i}^k x_{i,j} - \gamma_i
\le n \gamma_i + \frac{n+1}{3}.
\end{equation}
If $\gamma_i >0$, letting $k=1$ in equation~\eqref{ecua:prueba-kos},
we get
\begin{equation*}
\textstyle
x_{i,i} - \gamma_i \ge n \gamma_i - \frac{n+1}{3} \ge \frac{2n-1}{3} >0.
\end{equation*}
Thus, $X-C_\gamma$ satisfies (Po).
It remains to prove (CS) for $1 \le i \le k < \ell$, when $\gamma_i > \gamma_{i+1}$.
Applying~\eqref{ecua:prueba-kos} twice, for $i$ and $i+1$, we obtain
\begin{align*}
\textstyle
\sum_{j=i}^k x_{i,j} - \gamma_i & \textstyle \ge n \gamma_i - \frac{n+1}{3} \\
& \textstyle \ge n \gamma_{i+1} + n - \frac{n+1}{3} \ge n \gamma_{i+1} + \frac{n+1}{3} \\
& \textstyle \ge \sum_{j=i+1}^{k+1} x_{i+1,j} - \gamma_{i+1}.
\end{align*}
We have proved that $X - C_\gamma$ is in $\gtpat{\nu + n \gamma}{\rho + n \gamma}$.
Therefore $f_n$ is surjective.
The proof is complete.
\end{proof}

\begin{proof}[Proof of Theorem~\ref{teor:m}]
Let $A\in \dmatapi$ be an additive matrix, and let $\la$, $\mu$, $\nu$ be partitions
of the same size.
For $k\in \noneg$, denote
$s_k = \coefili{\la + k \alpha}{\mu + k \beta}{\nu + k \gamma}$.
If $s_k = 0$, for all $k$, there is nothing to prove.
So, we assume that there is some $n\in \noneg$ such that $s_k >0$, and denote by $m$
the smallest $k$, such that $s_k >0$.
Since $A$ is additive, Theorem~\ref{teor:adi-uni-kron} implies
that $\coefili \alpha \beta \gamma = 1$.
Then, by Theorem~\ref{teor:kron-mani} and induction on $k$, we have that
$s_{k+1} \ge s_k$, for all $k \ge m$.
It remains to show that the sequence $\{ s_k \}_{k\in\noneg}$ is bounded from above.
By equations~\eqref{ecua:kron} and~\eqref{ecua:young}, one has
\begin{equation*}
s_k = \bili{\cara{\la + k \alpha} \otimes \cara{\mu + k \beta}}{\cara{\nu + k \gamma}}
\le \bili{\permu{\la + k \alpha} \otimes \permu{\mu + k \beta}}{\cara{\nu + k \gamma}}.
\end{equation*}
Therefore, equation~\eqref{ecua:permu-mat} yields
\begin{equation} \label{ecua:suma-kos}
s_k \le \sum_{X \in \dmatli{\la + k \alpha}{\mu + k \beta}}  K_{\nu + k\gamma, \parti X}.
\end{equation}
Condition~\eqref{ecua:kos-pos} implies that the sum in equation~\eqref{ecua:suma-kos}
runs over all $X$ such that,
\begin{equation*}
X \in \fimapinv{\permuta{\nu + k \gamma}} \cap \dmatli {\la + k \alpha}{\mu + k \beta}.
\end{equation*}
Let $m \in \natural$ be as in Remark~\ref{obse:pt-acota}, and write
\begin{equation*}
\fimapinv{\permuta{\nu + m \gamma}} \cap \dmatli {\la + m \alpha}{\mu + m \beta} =
\{ X_1, \dots, X_l \}.
\end{equation*}
Let $\rho(j) = \parti{X_j}$ and $k \ge m$.
By Proposition~\ref{prop:pt-acota}, the number of non-zero summands
in~\eqref{ecua:suma-kos} is $l$, and, by Remark~\ref{obse:pt-acota},
for any $j\in \nume l$, $\parti{X_j + (k-m)A} = \rho(j) + (k-m)\gamma$.
By Proposition~\ref{prop:k-acota}, for each $j\in \nume l$, there is $m(j)\ge m$,
such that, the sequence
\begin{equation*}
\{ K_{\nu + m\gamma + (k-m)\gamma,\, \rho(j) + (k-m)\gamma} \}_{k\ge m(j)}
\end{equation*}
is constant.
Let $N$ bigger than $m$ and all $m(j)$'s.
Let $k \ge N$.
Then, the sum in the right side of inequality~\eqref{ecua:suma-kos} is equal to
\begin{equation*} \textstyle
\sum_{j \in \nume l}  K_{\nu + k\gamma,\, \rho(j) + (k-m)\gamma},
\end{equation*}
and this number is the same for all $k \ge N$.
Hence, the sequence $\{ s_k \}_{k\in \natural}$ is bounded, and the
theorem is proved.
\end{proof}

\section{Relation between Stembridge's hypothesis and additivity} \label{sec:equiv}
\hfill

In this Section we show that Stembridge's hypothesis~\eqref{ecua:stem} and
additivity are closely related.
On the one hand, additivity implies condition~\eqref{ecua:stem},
on the other, this condition implies the existence of a certain unique additive matrix.
We also give a new characterization of additivity (Theorem~\ref{teor:adi-n-uni}).

\begin{teor} \label{teor:adi-stem}
Let $A \in \dmatapi$ be an additive matrix.
Then, $\bili{\permu{n\alpha} \otimes \permu{n\beta}}{\cara{n\gamma}} = 1$,
for all $n \in \natural$.
\end{teor}
\begin{proof}
Since $A$ is additive, $nA$ is additive.
Then, by Theorem~\ref{teor:adi-pit}, we have
\begin{equation*}
\permuta {n\gamma} \cap \fimap {\transport {n\alpha}{n\beta}} = \{ \fimap {nA} \}.
\end{equation*}
In other words, $nA$ is the only matrix $X \in \dmatli {n\alpha}{n\beta}$ satisfying
$\parti X \menori n\gamma$.
Besides, since $\parti{nA} = n\gamma$, $K_{n\gamma, \parti{nA}} =1$.
Then, the theorem follows from identity~\eqref{ecua:permu-mat}.
\end{proof}

An alternative proof of this Theorem can be given, using Theorem~\ref{teor:a-mu} and
identity~\eqref{ecua:permu-mat}.
If $\bili{\permu{n\alpha} \otimes \permu{n\beta}}{\cara{n\gamma}} = 1$,
for all $n \in \natural$, we cannot assure the existence of an additive matrix
in $\dmatapi$ (see Example~\ref{ejem:equiv-falla}).
But we can prove.

\begin{teor} \label{teor:as}
Let $\alpha$, $\beta$, $\gamma$ be partitions of the same size.
If $\bili{\permu{n\alpha} \otimes \permu{n\beta}}{\cara{n\gamma}} = 1$,
for all $n \in \natural$.
Then, there is a unique matrix $A \in\permuta \gamma \cap \dmata$.
Besides $A$ is additive and $\kos {n\gamma,}{\parti A} =1$.\footnote{Compare
with Theorem~6.1 in~\cite{stem1}.}
\end{teor}
\begin{proof}
Let $n\in \natural$.
Then, by identity~\eqref{ecua:permu-mat}, there is only
one matrix $X_n\in \dmatli {n\alpha}{n\beta}$ with $\parti{X_n} \menori n\gamma$.
And, for this unique matrix, one has $K_{n\gamma, \parti {X_n}} =1$.
Let $A = X_1$.
Then $\parti A \menori \gamma$, and $\parti{nA} \menori n\gamma$.
Since $nA \in \dmatli {n\alpha}{n\beta}$, then $nA$ must be equal to $X_n$.
We claim that $A$ is additive.
By Theorem~\ref{teor:adi-pit}, it is enough to prove
\begin{equation} \label{ecua:rtf}
\permuta {\parti A} \cap \fimap {\transport \alpha \beta} = \{ \fimap A \}.
\end{equation}
Let $V$ be a matrix such that $\fimap V$ is any vertex of
$\permuta {\parti A} \cap \fimap {\transport \alpha \beta}$.
Due to the nature of the defining inequalities of this polytope,
$\fimap V$ has rational coordinates.
Therefore, there is some natural number $m$, such that $m \fimap V$ has integer
coordinates.
But $m \fimap V$ is in $\permuta {m\parti A} \cap \fimap {\transport {m\alpha}{m\beta}}$.
The uniqueness of $X_m$ implies $mV = mA$.
Thus, $\fimap A = \fimap V$.
This means that $\fimap A$ is the only vertex in
$\permuta {\parti A} \cap \fimap {\transport \alpha \beta}$.
Then, equation~\eqref{ecua:rtf} holds, and the theorem is proved.
\end{proof}

\begin{ejem} \label{ejem:equiv-falla}
Let $\alpha = (7,1)$, $\beta = (5,3)$ and $\gamma = (4,4)$.
We will show that, for all $n\in \natural$,
$\bili{\permu{n\alpha} \otimes \permu{n\beta}}{\cara{n\gamma}} = 1$, and that
there is no additive matrix in $\dmatapi$.
First, note that $\dmatapi = \vacio$.
Hence, the second assertion holds.
For each $t\in \natural$, let $X_ t = \matricita {4n+t}{3n-t}{n-t}t$.
Then,
\begin{equation*}
\dmatli {n\alpha}{n\beta} = \{ X_t \mid 0\le t \le n \}.
\end{equation*}
The only $t$ for which, $\parti{X_t} \menori n \gamma$, is $t=0$.
Since, $\kos {n\gamma,}{n(4,3,1)} = 1$, condition~\eqref{ecua:kos-pos}
and identity~\eqref{ecua:permu-mat} imply the first assertion.
\end{ejem}

The following result is implicit in the proofs of Theorems~\ref{teor:adi-stem}
and~\ref{teor:as}.

\begin{teor} \label{teor:adi-n-uni}
Let $A$ be a plane partition.
Then $A$ is additive if and only if $nA$ is minimal and $\pi$-unique,
for all $n\in \natural$.
\end{teor}
\begin{proof}
If $A$ is additive, then $nA$ is additive, for all $n\in \natural$.
The result follows from Threorem~\ref{teor:a-mu}.
For the converse, let $\gamma = \parti A$, and assume that,
for all $n\in \natural$, $nA$ is minimal and $\pi$-unique.
Then, $\bili{\permu{n\alpha} \otimes \permu{n\beta}}{\cara{n\gamma}} = 1$,
for all $n \in \natural$.
Since $A$ is the unique matrix in $\permuta\gamma \cap \dmata$,
Theorem~\ref{teor:as} implies that $A$ is additive.
\end{proof}

\begin{ejem}
Let $A$ and $X$ be defined as in Example~\ref{ejem:uni-si-adi-no}.
We now that $A$ is minimal and $\pi$-unique, but not additive (Example~\ref{ejem:uni-no-adi}).
Besides $2A$ is not minimal because
\begin{equation*}
\parti{2A -X} = (10^5,9^2,7^2,6^2,5^2,3^2,2^4,1^2) \menore
(10^6,8^2,6^4,4^2,2^6)= \parti{2A}.
\end{equation*}
\end{ejem}

\section{Relations between additive stability and other stability known results} \label{sec:aplica}
\hfill

In this section we explore the relation of previous stability known results to
additive stability.

Let us consider the additive triples $((|\alpha|), \alpha, \alpha)$ and
$(\alpha, \alpha^\prime, (1^{|\alpha|}))$, see Example~\ref{ejem:adi-tri}.
Practically, all known stability properties follow from Theorem~\ref{teor:m}
applied to some instance of one of these triples, together with some
symmetry of Kronecker coefficients.
In other words, every known stability property is related, up to
symmetry of its graph (Paragraph~\ref{para:sime}), to some particular case
of the additive matrix from Example~\ref{ejem:adi-tri}, or
equivalently, to some particular case of the additive matrix from
Example~\ref{ejem:young}.
The power of additive stability is that we can construct very easily many other
examples of additive matrices, thus producing new instances of stability.
However, the method of proof of this theorem does not produce an
explicit bound $L$ for stability.

Murnaghan's stability~\eqref{ecua:mur} follows from $\alpha = (1)$.
In this case, there are several known bounds $L$ for stability~\cite{bor, bri, vestab1, vdiag}.

The stability from equation~\eqref{ecua:mur-mas-gen-bis}, which includes~\eqref{ecua:mur-gen}
as a particular case, follows, in case $n_1 = \cdots = n_t$, from Theorem~\ref{teor:m}
applied to the additive triple $((|\alpha|), \alpha, \alpha)$ together with
the symmetry $\coefili \zeta \eta \theta = \coefili \eta \theta \zeta$.
The diagrammatic method from~\cite{vdiag} provides better results in the case of the
triple $((|\alpha|), \alpha, \alpha)$, because it gives very precise bounds for
stability and also because it permits to consider independent parameters
$n_1, \dots, n_t$ in equation~\eqref{ecua:mur-mas-gen-bis}.

Finally, the stability property from equation~\eqref{ecua:estab-cubi}
follows from Theorem~\ref{teor:m} applied to the additive triple
$(\alpha, \alpha^\prime, (1^{|\alpha|}))$ in the particular case
in which $\alpha = (q^r)$ is a rectangular partition, together with a
symmetry of Kronecker coefficients.\footnote{Compare with Proposition~A2
in~\cite{stem2} and Theorem~4.6 in~\cite{estab2}.}
Formula~\eqref{ecua:estab-cubi} gives, under some conditions, a very good bound
for stability, which is useful in some applications (see~\cite[\S5]{estab2}).

Note that for arbitrary $\alpha$ the additive triple
$(\alpha, \alpha^\prime, (1^{|\alpha|}))$ yields a new stability
property (see Corollary~\ref{coro:adi-young}).

\vskip 1.5pc
\textbf{Acknowledgements.}
I would like to thank to Richard Brualdi, Jes\'us De Loera, Shmuel Onn, Igor Pak,
Greta Panova and Adolfo Torres-Ch\'azaro for fruitful conversations.


\vskip 2.5pc
{\small
\begin{thebibliography}{99}

\bibitem{avva} D.~Avella-Alaminos and E.~Vallejo,
Kronecker products and the RSK correspondence,
\emph{Discrete Math.}\ \textbf{312} (2012), 1476--1486.

\bibitem{bisa} L.J. Billera and A. Sarangarajan,
The combinatorics of permutation polytopes, in
``Formal Power Series and Algebraic Combinatorics 1994'', 1 -- 23
DIMACS Series in Discrete Mathematics and Theoretical Computer Science,
Vol. 24, L.J. Billera, C. Greene, R. Simion, R.P. Stanley eds. AMS 1996.

\bibitem{bor} E.~Briand, R.~C.~Orellana and M.~Rosas,
The stability of the Kronecker product of Schur functions,
\emph{J. Algebra}\ \textbf{331} (2011), 11--27.

\bibitem{bri} M. Brion,
Stable properties of plethysm: on two conjectures of Foulkes,
\emph{manuscripta math.}\ \textbf{80} (1993), 347--371.

\bibitem{bru} R. Brualdi,
Minimal nonnegative integral matrices and uniquely determined $(0,1)$-matrices,
\emph{Linear Algebra Appl.}\ \textbf{314} (2002), 351--356.

\bibitem{bdg} S. Brunetti, A. Del Lungo and Y. Gerard.
On the computational complexity for reconstructing three-dimensional
lattice sets from their two-dimensional $X$-rays,
\emph{Linear Algebra Appl.}\ \textbf{339} (2001), 59--73.

\bibitem{cheb} T. Church, J.S. Ellenberg and B. Farb.
FI-modules: a new approach to stability for $S_n$-representations,
{\tt arXiv:1204.4533v2 [math.RT]}, 28 Jun 2012.

\bibitem{delki} J. De Loera and E.D. Kim,
Combinatorics and geometry of transportation polytopes: an update,
{\tt arXiv:1307.0124v1}, 29 Jun 2013.

\bibitem{dfr} P. Doubilet, J. Fox and G.R. Rota,
\emph{The elementary theory of the symmetric group} in
``Combinatorics, representation theory and statistical methods in groups'',
T.V. Narayama, R.M. Mathsen and J.G. Williams eds.,
Lecture Notes in Pure and Applied Mathematics, Vol. 57,
Marcel Decker, New York, 1980.

\bibitem{dupe} P. Dulio and C. Peri.
Discrete tomography and plane partitions,
\emph{Adv. in Applied Math.}\ \textbf{50} (2013), 390--408.

\bibitem{flrs} P.C. Fishburn, J.C. Lagarias, J.A. Reeds and L.A. Shepp,
Sets uniquely determined by projections on axes II. Discrete case,
\emph{Discrete Math.}\ \textbf{91} (1991), 149--159.

\bibitem{fishe} P.C. Fishburn and L.A. Shepp,
Sets of uniqueness and additivity in integer lattices.
In: Herman,~G.T., Kuba,~A. (eds.),
\emph{Discrete Tomography: Foundations, Algorithms, and Applications}
Birkh\"auser, Boston, pp.~35--58 (1999).

\bibitem{gridev} P. Gritzmann and S. de Vries,
On the algorithmic inversion of the discrete Radon transform,
\emph{Theor. Comp. Science}\ \textbf{281} (2002), 455--469.

\bibitem{halipo} G. Hardy, J.E. Littlewood and G. P\'olya,
``Inequalities'',
Second ed., Cambridge Univ. Press, 1952.

\bibitem{heka1} G.T. Herman and A. Kuba, editors,
``Discrete Tomography: Foundations, Algorithms, and Applications'',
Birkh\"auser, Boston, 1999.

\bibitem{heka2} G.T. Herman and Attila Kuba, editors,
``Advances in Discrete Tomography and Its Applications'',
Applied and Numerical Harmonic Analysis,
Birkh\"auser, Boston, 2007.

\bibitem{jake} G.D. James and A. Kerber,
``The representation theory of the symmetric group'',
Encyclopedia of mathematics and its applications, Vol. 16,
Addison-Wesley, Reading, Massachusetts, 1981.

\bibitem{lit2} D.E. Littlewood,
Products and plethysms of characters with orthogonal, symplectic and symmetric groups,
\emph{Canad. J. Math.}\ \textbf{10} (1958), 17-–32.

\bibitem{mac} I.G. Macdonald,
``Symmetric functions and Hall polynomials'', 2nd. edition
Oxford Mathematical Monographs Oxford Univ. Press 1995.

\bibitem{mani} L. Manivel,
On rectangular Kronecker coefficients,
\emph{J. Algebr. Comb.}\ \textbf{33} (2011), 153--162.

\bibitem{maol} A.W. Marshall, I. Olkin,
``Inequalities: Theory of majorization and its applications'',
Academic Press, 1979.

\bibitem{mun} J.R. Munkres,
``Topology, a first course'',
Prentice-Hall, Englewood Cliffs, 1975.

\bibitem{mur} F.D. Murnaghan,
The analysis of the Kronecker product of irreducible representations
of the symmetric group,
\emph{Amer. J. Math.}\ \textbf{60} (1938), 761--784.

\bibitem{mur2} F.D. Murnaghan,
On the analysis of the Kronecker product of irreducible representations of $S_n$,
\emph{Proc. Nat. Acad. Sci. USA}\ \textbf{41} (1955), 515--518.

\bibitem{onva} S. Onn and E. Vallejo,
Permutohedra and minimal matrices,
\emph{Linear Algebra Appl.}\ \textbf{412} (2006), 471--489.

\bibitem{pp3} I. Pak and G. Panova,
Bounds on the Kronecker coefficients,
{\tt arXiv:1406.2988v2}, 16 Jun 2014.

\bibitem{pava} I. Pak and E. Vallejo,
Combinatorics and geometry of Littlewood-Richardson cones,
\emph{Europ. J. Comb.}\ \textbf{26} (2005), 995--1008.

\bibitem{ra} R. Rado,
An inequality,
\emph{J. London Math. Soc.}\ \textbf{27} (1952), 1--6.

\bibitem{sag} B.E. Sagan,
``The Symmetric Group," 2nd. ed.,
Graduate Texts in Mathematics 203, Springer Verlag, 2001.

\bibitem{san} M. Santoyo,
Aditividad para matrices y $(0,1)$-aditividad para pirámides,
Ph. D Thesis, Universidad Michoacana de San Nicolás de Hidalgo, 2008.

\bibitem{sava} M. Santoyo and E. Vallejo,
Additivity obstructions for integral matrices and pyramids,
\emph{Theoret. Comput. Sci.}\ \textbf{406} (2008), 136--145.

\bibitem{stanspp} R.P. Stanley,
Symmetries of plane partitions,
\emph{J. Combin. Theory. Ser. A}\ \textbf{43} (1986), 103--113.

\bibitem{stan} R.P. Stanley,
``Enumerative Combinatorics, Vol. 2" ,
Cambridge Studies in Advanced Mathematics 62. Cambridge Univ. Press, 1999.

\bibitem{stem1} J.R. Stembridge,
Generalized stability of Kronecker coefficients,
preprint.

\bibitem{stem2} J.R. Stembridge,
Appendix to: Generalized stability of Kronecker coefficients,
preprint.

\bibitem{thi} J.Y. Thibon,
Hopf algebras of symmetric functions and tensor products of symmetric group representations,
\emph{Internat. J. Algebra Comput.}\ \textbf{1} (1991), 207--221.

\bibitem{tova} A. Torres-Ch\'azaro and E. Vallejo,
Sets of uniqueness and minimal matrices.
\emph{J. Algebra}\ \textbf{208} (1998), 444--451.

\bibitem{vapir} E. Vallejo,
Reductions of additive sets, sets of uniqueness and pyramids,
\emph{Discrete Math.}\ \textbf{173} (1997), 257--267.

\bibitem{vestab1} E. Vallejo,
Stability of Kronecker product of irreducible characters of the
symmetric group,
\emph{Electron. J. Combin}\ \textbf{6} (1999) Reseach Paper 39, 7 pp.

\bibitem{vapp} E. Vallejo,
Plane partitions and characters of the symmetric group,
\emph{J. Algebraic Comb.}\ \textbf{11} (2000), 79--88.

\bibitem{aditivo} E. Vallejo,
A characterization of additive sets,
\emph{Discrete Math.}\ \textbf{259} (2002), 201--210.

\bibitem{minidosq} E. Vallejo,
The classification of minimal matrices of size $2 \times q$,
\emph{Linear Algebra Appl.}\ \textbf{340} (2002), 169--181.

\bibitem{elenot} E. Vallejo,
Minimal matrices and discrete tomography,
\emph{Electron. Notes Discrete Math.}\ \textbf{20} (2005), 113--132.

\bibitem{libro} E. Vallejo,
Uniqueness and additivity for $n$-dimensional  binary matrices
with respect to their $1$-marginals.
In: Herman,~G.T., Kuba,~A. (eds.),
\emph{Advances in Discrete Tomography and its Applications},
Birkh\"auser, Boston, pp.~83--112 (2007).

\bibitem{estab2} E. Vallejo,
A stability property for coefficients in
Kronecker products of complex $S_n$ characters,
\emph{Electron. J. Combin}\ \textbf{16} (2009) \#N22, 8 pp. (electronic).

\bibitem{vdiag} E. Vallejo,
A diagrammatic approach to Kronecker squares,
\emph{J. Combin. Theory. Ser. A}\  \textbf{127} (2014), 243--285.

\bibitem{ykk} V.A. Yemelichev, M.M. Kovalev and M.K. Kravtsov,
``Polytopes, Graphs and Optimisation'',
Cambridge University Press, Cambridge, 1984.

\end{thebibliography}
}
\end{document}